\documentclass[a4paper,12pt]{article}
\usepackage{amssymb,amsthm,amsmath,amsfonts}
\usepackage{graphicx}
\usepackage{cite}
\usepackage{color}

\theoremstyle{definition}
\newtheorem{theorem}{Theorem}

\newtheorem{corollary}[theorem]{Corollary}
\newtheorem{proposition}[theorem]{Proposition}
\newtheorem{definition}[theorem]{Definition}

\theoremstyle{remark}
\newtheorem{remark}{Remark}

\newcommand{\bB}{{\mathbb{B}}}

\newcommand{\bN}{{\mathbb{N}}}

\newcommand{\bR}{{\mathbb{R}}}

\DeclareMathOperator{\spn}{span}

\begin{document}

\title{The Borel-Pompieu formula involving proportional fractional $\psi$-Cauchy-Riemann operators}
\small{
\author
{Jos\'e Oscar Gonz\'alez-Cervantes$^{(1)}$, Isidro Paulino-Basurto$^{(1)}$\\ and\\ Juan Bory-Reyes$^{(2)\footnote{corresponding author}}$}
\vskip 1truecm
\date{\small $^{(1)}$ Departamento de Matem\'aticas, ESFM-Instituto Polit\'ecnico Nacional. 07338, Ciudad M\'exico, M\'exico\\ Email: jogc200678@gmail.com, isidroo-@hotmail.com\\$^{(2)}$ {SEPI, ESIME-Zacatenco-Instituto Polit\'ecnico Nacional. 07338, Ciudad M\'exico, M\'exico}\\Email: juanboryreyes@yahoo.com
}}

\maketitle
\begin{abstract}
We prove an analog of the quaternionic Borel-Pompieu formula in the sense of proportional fractional $\psi$-Cauchy-Riemann operators via Riemann-Liouville derivative with respect to another function. 
\end{abstract}

\noindent
\textbf{Keywords:} Borel-Pompieu formula, $\psi$-Cauchy-Riemann operators, Proportional fractional integrals and derivatives, Fractional calculus with respect to functions.\\
\textbf{AMS Subject Classification (2020):} 26A33; 30G30; 30G35; 45P05.

\section{Introduction} 
The main purpose of this paper is to bring together the proportional fractional calculus with respect to another function and the theory of $\psi$-hyperholomorphic functions in the sense of $\psi$-Cauchy-Riemann operators.

Quaternionic analysis in $\mathbb R^3$ is described as a theory of quaternionic algebra-valued functions of three real variables based on the null-solutions of the so-called $\psi$-Moisil-Theodorescu (or $\psi$-Dirac) operator (a generalization of the complex holomorphic functions) offering a refinement of classical harmonic analysis, where $\psi$ is a structural set different from the standard basis vectors of $\mathbb R^3$ \cite{MS, S, SV1, SV2, VS}. 

In recent decades, a number of papers \cite{AP, BGL, GN1, GN2, GN3, GNL, JSK, Ng, MP, Pi} appeared concerning a $\psi$-hyperholomorphic functions theory in the sense of $\psi$-Cauchy-Riemann operators. 

Fractional calculus, which is engaged with derivatives and integrals of non-integer order, has emerged as important area of investigation in several branches of science and engineering, see \cite{KST, OS, O, P, MR, SKM}. In particular, the interest for  using the notion of proportional derivatives of a function with respect to another function, is addressed in \cite{A, KST, JARH, JAAl, JAA, JUAB, Na, NAAA, OFDF} and the references given there. 

The development of a fractional hyperholomorphic function theory (for functions with values in the quaternions or more in general in a Clifford algebra) is a recent field of research, see \cite{GB-1, GB-2, GB-3, GB-4, CTOP, DM, FRV, KV, PBBB, V, ZT} and the references therein.

Inspired by the mentioned works, which are not intended to be a full references list on the addressed themes, this paper aims to establish a new proportional fractional $\psi$-Cauchy-Riemann type operator calculus with respect to an $\mathbb R$-valued function via Riemann-Liouville fractional derivative in quaternionic analysis context. A useful tool proven here is a new quaternionic Borel-Pompieu type formula that involves proportional fractional $\psi$-Cauchy-Riemann operators with respect to $\mathbb R$-valued functions. Our approach will be different from the works which has been done in \cite{GB-1, GB-2, GB-3, GB-4}.

After this brief introduction, the present work proceeds as follows: Preliminary's section contains a summary of basic facts about quaternionic analysis, proportional fractional derivatives and their main properties. Section 3 deals with the notion of
proportional fractional $\psi$-hyperholomorphic functions and some results. In Section 4 the major achievements are stated and proved.

\section{Preliminaries} 
In this section we in brief present a summary of basic concepts of fractional proportional operators as well as rudiments of quaternionic analysis employed throughout the work.
\subsection{Proportional fractional derivatives with respect to functions}
Given $0<\rho< 1$ and $f\in C^1(\mathbb R, \mathbb R)$ then the proportional derivative  and integral of $f$ of proportional associated to  $\rho\in [0,1]$ are, respectively, 
\begin{align*} (D^{\rho}f )(t) = & (1-\rho)f (t) +  \rho f'(t), \\
({}_aI^{1-\rho} f) (t) = &\frac{1}{\rho}\int_
a^t  e^{ \frac{\rho-1}{\rho}(t-s)}f (s) ds, \nonumber  
 \end{align*}
where ${}_aI^{0}f (t) = f (t)$. In addition, if $\varphi\in C^1(\mathbb R, \mathbb R^+ )$, where $\mathbb R^+$ denotes the set of all real positive numbers, then 
\begin{align*} 
(D^{\rho,\varphi} f )(t) = (1-\rho)f(t) + \rho\frac{f'(t)}{\varphi'(t)}
\end{align*}
is proportional derivative of $f\in C^1(\mathbb R,\mathbb R)$ with respect to $\varphi$.  The proportional integral of $f$  with respect to $\varphi$  and order $n \in \mathbb N$ is considered similarly:
\begin{align*}
 ({}_a I^{n,\rho, \varphi} f) (t) = \frac{1}{\rho^n \Gamma(n)}
 \int_a^t
e^{\frac{\rho-1}{\rho} (\varphi(t)-  \varphi(\tau)) }
(\varphi(t)- \varphi(\tau))
^{n-1}f (\tau)\varphi'(\tau) d\tau.
\end{align*}
From direct computations we see that
\begin{align*} 
\frac{\rho}{\varphi'(t)}\frac{\partial }{\partial t}(e^{\varphi(t) \rho^{-1}(1-\rho)   }f(t)) = & e^{\varphi(t) \rho^{-1}(1-\rho)  }\left(  (1-\rho)f(t) +\rho  \frac{f'(t)}{\varphi'(t) }  \right)  \nonumber \\
=& e^{\varphi(t) \rho^{-1}(1-\rho)  } (D^{\rho,\varphi} f )(t).
\end{align*}

Set $\alpha\in \mathbb C$, with $0< \Re \alpha <1 $.  The left and the right fractional  proportional integrals with respect to $\varphi$, of order $\alpha$ are defined by  
\begin{align*}  ({}_a I^{\alpha,\rho,\varphi} f) (t) := &\frac{1}{\rho^\alpha \Gamma(\alpha)} \int_a^te^{\frac{\rho-1}{\rho} (\varphi(t)-\varphi(\tau)) }
(\varphi(t)- \varphi(\tau))
^{\alpha-1}f (\tau)\varphi'(\tau) d\tau , \nonumber \\
(I_b^{\alpha,\rho,\varphi} f)(t) :=& \frac{1}{\rho^\alpha \Gamma(\alpha)}
 \int_ t^b
e^{\frac{\rho-1}{\rho} (\varphi(\tau)- \varphi(t)) }
(\varphi (\tau)-  \varphi (t))
^{\alpha-1}f (\tau)\varphi'(\tau) d\tau, 
\end{align*}
respectively. Meanwhile, the left and the right fractional proportional derivatives with respect to $\varphi$ are
\begin{align}\label{DerFracPropRespecFuntL}
({}_aD^{\alpha, \rho,\varphi} f)(t) :=  D^{\rho,\varphi}
{}_aI^{1-\alpha,\rho,\varphi}f(t) \quad \textrm{and} \quad
(D_b^{\alpha, \rho,\varphi} f)(t) := D^{\rho,\varphi}
 I_b^{1-\alpha,\rho,\varphi}f(t).
\end{align}
An important property of these fractional differential and integral operators is contained in the so-called the fundamental theorem:
\begin{align}\label{FundTheoFPW}
{}_aD^{\alpha,\rho,\varphi} \circ 
{}_aI^{\alpha,\rho,\varphi} f(t) = f(t) \quad \textrm{and} \quad
 D_b^{\alpha,\rho,\varphi} \circ 
 I_b^{\alpha,\rho,\varphi} f(t) = f(t).
\end{align}

An important special case is when $\varphi(t)$ is the identity function, denoted by $I(t)$, and $\rho=1$ to have 
\begin{align*} ({}_a I^{\alpha,1,I} f) (t)  = 
&\frac{1}{  \Gamma(\alpha)} \int_a^t 
( t-\tau )
^{\alpha-1}f (\tau)  d\tau =: ({}_a I^{\alpha} f) (t),
\end{align*}
which is the well-known left-fractional integral of $f$ in the Riemann-Liouville sense and, similarly, 
$
(I_b^{\alpha,1,I} f)(t) = (I_b^{\alpha} f)(t)$. What is more, 
\begin{align*}
({}_aD^{\alpha, 1,I} f)(t)  =  D^{1,I}
{}_aI^{1-\alpha,1,I}f(t) = \frac{\partial}{\partial t}{}_aI^{1-\alpha}f(t)=:
({}_aD^{\alpha} f)(t), 
\end{align*}
is the left-fractional partial of $f$ of order $\alpha$ and 
$(D_b^{\alpha, 1,I} f)(t)= (D_b^{\alpha} f)(t) $ is the right
left-fractional partial of $f$ of order $\alpha$, both in the Riemann-Liouville sense.

\subsection{On $\psi$-Cauchy-Riemann operator calculus in $\mathbb R^3$}
Let $\mathbb H$ denotes the skew-field of quaternions whose standard  basis is $\{{\bf e_0}, {\bf e_1}, {\bf e_2}, {\bf e_3}\}$, ${\bf e_0}$ is the unit and $\{{\bf e_1}, {\bf e_2}, {\bf e_3}\}$ are called imaginary units.

Any $q\in \mathbb H$ is given by
$$q=x_0 {\bf e_0}+x_{1} {\bf e_1}+x_{2} {\bf e_2}+x_{3} {\bf e_3},$$ 
where $x_{k}\in \mathbb R, k= 0,1,2,3$. The sum is between coefficients of basic elements and the multiplication obeys the following rules
$$ {\bf e_1}^{2}= {\bf e_2}^{2}= {\bf e_3}^{2}=-{\bf e_0},$$
$$ {\bf e_1}\, {\bf e_2}=- {\bf e_2}\,\  {\bf e_1}= {\bf e_3};\   {\bf e_2}\, {\bf e_3}=- {\bf e_3}\, {\bf e_2}= {\bf e_1};\  {\bf e_3}\, {\bf e_1}=- {\bf e_1}\, {\bf e_3}= {\bf e_2}.$$ 
The real span of $\{\bf e_0, \bf e_1, \bf e_2\}$, denoted by $\mathcal A:=\spn_{\mathbb R}\{\bf e_0, \bf e_1, \bf e_2\}$, forms a real subset of $\mathbb H$ but not a sub algebra.

The main involution in $\mathbb H$, called quaternionic conjugation, is given by 
$$q\rightarrow {\overline q}:=x_0 {\bf e_0}-x_{1}{\bf e_1}-x_{2}{\bf e_2}-x_{3}{\bf e_3}$$ 
and the norm can be introduced as 
$$|q|  := \sqrt{ x^{2}_{0}+x^{2}_{1}+x^{2}_{2}+x^{2}_{3}}= \sqrt{ q\,{\overline q}}=\sqrt{{\overline q}\,q}.$$
Given  $u, v\in\mathbb H$ its scalar product is 
$$\langle u, v\rangle:=\frac{1}{2}(\bar u v + \bar v u).$$

An ordered set $\psi:=\{\psi_0, \psi_1,\psi_2\}\subset \mathcal A$ which satisfies the orthonormality condition  
$$\psi_k{\overline \psi}_s+{\overline\psi}_s\psi_k=\langle \psi_k, \psi_s\rangle =2\delta_{k,s},$$ 
for $k,s=0,1,2$, where $\delta_{k,s}$ is the Kronecker's symbol is called $\mathcal A$-structural set.

The real three-dimensional vector space $\mathbb R^3$ will be embedded in $\mathbb H$ by identifying the element $x= (x_0,x_1,x_2)\in \mathbb R^3$ with $x_0\psi_0 + x_1\psi_1+ x_2 \psi_2\in \mathcal A$.  

Let $\Omega\subset\mathbb R^3 \cong \mathcal A$ be an open bounded domain with boundary $\partial \Omega$ a smooth surface and by $\overline\Omega$ we mean its closure.  All $\mathcal A$-valued functions $f$ defined on $\Omega$ have the form $f=\sum_{k=0}^2 f_k \psi_k$, with $f_k:\Omega\to \mathbb R$, for $ k= 0,1,2$. The function $f$ is said to be continuously differentiable if every $f_k$ has the property. As usual, the corresponding space is denoted by $C^{1}(\Omega, \mathcal A)$.
  
We define the left and the right-$\psi$-Cauchy-Riemann operators acting on $f \in C^1(\Omega,\mathcal A)$ respectively as follows    
$${}^{{\psi}}\overline{\partial}f := \sum_{k=0}^2 \psi_k \partial_k f$$ 
and 
$${}^{{\psi}}\overline{\partial}_r f :=  \sum_{k=0}^2 \partial_k f \psi_k,$$ 
where we are using the partial differentiation operator $\partial_k f :=\displaystyle \frac{\partial f}{\partial x_k}$ for all $k$. 

These operators decompose the three-dimensional Laplace operator $\Delta_{\mathbb R^3}$ according to the real components of $x$. Indeed 
$${}^{{\psi}}\overline{\partial}\circ {}^{\overline{\psi}}\overline{\partial}={}^{\overline{\psi}}\overline{\partial} \circ {}^{{\psi}}\overline{\partial} = {}^{{\psi}}\overline{\partial}_r\circ {}^{\overline{\psi}}\overline{\partial}_r={}^{\overline{\psi}}\overline{\partial}_r \circ {}^{{\psi}}\overline{\partial}_r =\bigtriangleup_{\mathbb R^3}.$$

A function $f$ is say to be left-$\psi$-hyperholomorphic (right-$\psi$-hyperholomorphic) on $\Omega$ if ${}^{{\psi}}\overline{\partial}[f]=0$ (${}^{{\psi}}\overline{\partial}_r[f]=0$) on $\Omega$.

The quaternionic differential form: 
$$\displaystyle  {}^{\psi}\eta_{x}:= \sum_{k=0}^2 (-1)^k \psi_k d\hat{x}_k$$ 
is associated to the $2$-dimensional  area differential in $\mathbb  R^3$, where $d\hat{x}_i  = dx_0 \wedge dx_1 \wedge dx_2 $ omitting  $dx_i$ and $ |\eta_{x}| = dx$ is the differential of $2$-dimensional  area. 

Given $f,g  \in C^1(\Omega,\mathcal A)$,  the differential Stokes theorem:
\begin{align*}
d(g{}^{\psi}\eta_x f)  = \left(g {}^{\psi}\overline{\partial} [f] + {}^{\psi}\overline{\partial}_r [g] f\right)dx,
\end{align*}
can be obtained from a direct computation as we see below: 
	\begin{align*}
d(g{}^{\psi}\eta_x f) = &   \sum_{k,l,m} d(g_k (-1)^l d\hat{x}_l  f_m) \psi_k\psi_l\psi_m \\
=&  \sum_{k,l,m} d(g_k   f_m) (-1)^l d\hat{x}_l \psi_k\psi_l\psi_m \\
=&  \sum_{k,l,m}  \left(\frac{\partial g_k}{\partial x_l}   f_m  +  g_k \frac{\partial f_m}{\partial x_l}  \right) (-1)^l dx_l d\hat{x}_l \psi_k\psi_l\psi_m \\
=&  \sum_{k,l,m}  \left(\frac{\partial g_k}{\partial x_l}   f_m  +  g_k \frac{\partial f_m}{\partial x_l}  \right) dx \psi_k\psi_l\psi_m .
\end{align*}
Thus the integral Stokes theorem induced by ${}^{\psi}\overline{\partial}$ and ${}^{\psi}\overline{\partial}_r$ can be state as follow
\begin{align}\label{StokesT}
\int_{\partial \Omega} g{}^{\psi}\eta_x f =  \int_{\Omega} \left[g({}^{\psi}\overline{\partial}  f)  + ({}^{\psi}\overline{\partial}_r  g) f\right]dx.
\end{align}
The ${\psi}$-hyperholomorphic Cauchy-Riemann kernel is given by
\[K_{\psi}(x) = \dfrac{x}{4\pi |x|^{3}},\]
for all $  x=x_0+\psi_1x_1+\psi_2x_2 \in\bR^{3}\setminus  \{0\}$ and, directly, one can verify its basic properties: $K_{\psi}\in C^{\infty} (\mathbb  R^3\setminus \{ 0\})$ and  ${}^{\psi}\overline{\partial}_{r} K_{\psi} =  0 $ on $\mathbb  R^3\setminus \{  0\}$. What is more, if  $f,g\in C^{1}(\overline{\Omega},\mathcal A)$ the induced Borel-Pompieu formula holds:  
\begin{align}\label{BP}
&\int_{\partial\Omega }\left[ K_{\psi}(\tau -x){}^{\psi}\eta_{\tau} f(\tau)+g(\tau) {}^{\psi} \eta_{\tau} K_{\psi}(\tau -x)\right] \nonumber \\
&-\int_{\Omega }\left[ K_{\psi}(y-x) ({}^{\psi}\overline{\partial}f)(y)+({}^{\psi}\overline{\partial}_{r}g)(y) K_{\psi}(y-x)\right] dy \nonumber  \\
=&
\begin{cases}
f(x)+g(x), & x\in\Omega 
\\
0, & x\in\bR^{3}\setminus\overline{\Omega},
\end{cases}
\end{align} 
Let $x=\displaystyle\sum_{k=0}^{2}\psi_k x_k \in \mathbb R^3 \setminus \Omega$, then combining the properties of $K_{\psi}$ and integral Stokes theorem \eqref{StokesT}, the previous formula follows. 

On the other hand, since for $x\in\Omega $ there exists $N\in\bN$ such that $\bB (x,\displaystyle\frac{1}{n})\subset\Omega $ for all  $n>N$, integral  Stokes theorem allows to see that 
\begin{align*}
&\int_{\partial (\Omega  \setminus\bB (x, \frac{1}{n}))} K_{\psi}(\tau -x) {}^{\psi}\eta_{\tau} f(\tau)  \\
&=  \int_{\Omega  \setminus\bB (x, \frac{1}{n})} \left[K_{\psi}(y-x) ({}^{\psi}\overline{\partial}f)(y) + ({}^{\psi}\overline{\partial}_{r}K_{\psi}(y-x))f(y)\right]dy ,
\end{align*}
or equivalently 
\begin{align*}
&\int_{\partial \Omega } K_{\psi}(\tau -x) {}^{\psi}\eta_{\tau} f(\tau) -
\int_{ \bB (x, \frac{1}{n}) } K_{\psi}(\tau -x) {}^{\psi}\eta_{\tau} f(\tau) \\
& =  \int_{\Omega  \setminus\bB (x, \frac{1}{n})}  K_{\psi}(y-x) {}^{\psi}(\overline{\partial}f)(y) dy .
\end{align*}
The Stokes theorem gives us that 
 \begin{align*}
\int_{\partial\bB (x, \frac{1}{n})} \overline{(\tau -x)}{}^{\psi}\eta_{\tau}  =  4\pi\dfrac{1}{n^{3}}.
\end{align*}
and the inequality:
\begin{align*}
&|\int_{\partial\bB (x, \frac{1}{n})} K_{\psi}(\tau -x){}^{\psi}\eta_{\tau} f(\tau) - f(x)| = |\dfrac{n^{3}}{4\pi}\int_{\partial\bB (x, \frac{1}{n})} \overline{(\tau -x)}{}^{\psi} \eta_{\tau} (f(\tau)-f(x))|  \\
&\leq \dfrac{n^{3}}{4\pi}\int_{\partial\bB (x, \frac{1}{n})}|\overline{\tau -x}| \, |{}^{\psi} \eta_{\tau}| \, |f(\tau )-f(x)| = \dfrac{n^{2}}{4\pi}\int_{\partial\bB (x, \frac{1}{n})}|{}^{\psi} \eta_{\tau}| \, |f(\tau )-f(x)|    
\end{align*}
allows to see that
\[\lim_{n\to \infty}
  \int_{\partial\bB (x, \frac{1}{n})} K_{\psi }(\tau -x){}^{\psi}\eta_{\tau}f(\tau )=f(x),\]
which shows  
\[ \int_{\partial\Omega } K_{\psi}(\tau -x){}^{\psi}\eta_{\tau} f(\tau)  -\int_{\Omega } K_{\psi}(y-x)({}^{\psi}\overline{\partial}f)(y) dy = f(x).\]
A similar identity is obtained for $g$ and ${}^{\psi}\overline{\partial}_rg$ and the main result is a sum of the previous two.

\section{Fractional proportional $\psi$-Cauchy-Riemann operators with respect to another function} 
Let $a:=\sum_{k=0}^2\psi_k a_k ,b:=\sum_{k=0}^2\psi_k b_k \in\mathbb R^3$ such that $a_k< b_k$, for all $k$. We introduce the following domain to be used mainly in the rest of the paper
$${J_a^b }:=  \{  \sum_{k=0}^2\psi_k x_k \in \mathbb R^3 \ \mid \ a_k< x_k < b_k, \  \ k=0,1,2\}= (a_0,b_0) \times (a_1,b_1) \times (a_2,b_2).$$ 
We temporary fix the element $y= \sum_{k=0}^2\psi_k y_k$ in $J_a^b$.

To define a fractional proportional $\psi$-Cauchy-Riemann operator with respect to $\mathbb R$-valued functions defined on ${J_a^b }$, let $\sigma:=\sum_{k=0}^2\psi_k \sigma_k $ and  ${\alpha}:= (\alpha_0, \alpha_1,\alpha_2), {\beta}:=(\beta_0, \beta_1, \beta_2) \in \mathbb C^3$ such that $0< \Re \alpha_k, \Re\beta_k$ and $\sigma_k < 1$ for all $k$. 

Let $\varphi\in C^1(\overline{J_a^b},\mathbb R)$ such that $\displaystyle\frac{\partial }{\partial x_i} \varphi(y_0, \dots, x_i, \dots y_2) > 0$ for all $x_i \in [a_i,b_i], i=0,1,2$. For abbreviation, let $D\varphi(x,y)$ stand for $\sum_{i=0}^2\displaystyle\frac{\partial }{\partial x_i} \varphi(y_0, \dots, x_i, \dots y_2)$. 

The real linear space $AC^1(J_a^b,\mathcal A)$ consists of functions ${f}=\sum_{i=0}^2\psi_i f_i$, where $f_i: J_a^b\to \mathbb R$ are required to be such that the mapping 
$$x_j \mapsto f_i(y_0,\dots,x_j,\dots, y_2)$$  
belongs to $AC^1((a_j, b_j), \mathbb R)$ for all $y\in J_a^b, i, j=0,1,2,$.  
 
\begin{definition} 
The left fractional proportional integral of $ f \in  AC^1(J_a^b,\mathcal A)  $ with respect to $\varphi$ with order $\alpha$ and proportion related to $\sigma$  is defined by
\begin{align*}
({}_a {  I}^{\alpha,\sigma,\varphi} f) (x,y): = &\sum_{i=0}^2 ({}_{a_i} I^{\alpha_i,\sigma_i,\varphi_i} f) (y_0, \dots, x_i ,\dots, y_2), 
\end{align*}
where $\varphi_i(x_i):= \varphi(y_0, \dots, x_i ,\dots, y_2)$ for all $x_i \in [a_i,b_i]$ for $i=0,1,2$. The right fractional proportional integral is     
  \begin{align*}
 ( {I}_b^{\alpha,\sigma,\varphi} f) (x,y) :=  \sum _{i=0}^2 (  I_{b_i}^{\alpha_i,\sigma_i,\varphi_i} f)(y_0, \dots, x_i ,\dots, y_3) . 
 \end{align*}
In addition, the quaternionic  left  and   right   fractional proportional $\psi$-Cauchy-Riemann operators  with respect to  $\varphi$ with order $\alpha$ and proportion related to $\sigma$ are given by    
\begin{align*}
({}^{\psi}_a{\overline{\partial}}^{\alpha, \sigma, \varphi }f)(x,y):= & (1-\sigma)  ({}_a{  I}^{1-\alpha, \sigma, \varphi}f)(x,y)  +
 \sigma \frac{{}^{\psi}\overline{\partial} ({}_a{I}^{1-\alpha,  \sigma, \varphi}  f )  (x,y)}{ D\varphi(x,y) }    
\end{align*} 
 and 
\begin{align*}
( {}^{\psi} \overline{\partial}_b^{\alpha, \sigma, \varphi }   f)  (x,y)  
:=  &
  (1-\sigma)  ( {I}_b^{1-\alpha,  \sigma, \varphi}  f )  (x,y) +
     \sigma \frac{ {}^{\psi}\overline{\partial} ({I}_b^{1-\alpha,  \sigma, \varphi}  f )  (x,y)}{ D\varphi(x,y) } ,   
\end{align*} 
respectively, where $1-\alpha := ( 1-\alpha_0, 1-\alpha_1, 1-\alpha_2 )\in \mathbb C^3 $ and the partials derivatives of ${}^{\psi}\overline{\partial}$ are with respect to real components of $x$.   

The right versions of the previous operators are the following: 
\begin{align*}
 ( {}^{\psi}_a\overline{\partial}_r^{\alpha, \sigma, \varphi } f)  (x,y)  =  & 
   ({}_a{ I}^{1-\alpha,  \sigma, \varphi}  f )  (x,y)   (1-\sigma)   +
\frac{    {}^{\psi}\overline{\partial}_r ({}_a{I}^{1-\alpha,  \sigma, \varphi}  f )  (x,q)  }{ D\varphi(x,y) }  \sigma, \\
 (  {}^{\psi}\overline{\partial}_{r,b}^{\alpha, \sigma, \varphi }    f )  (x,y)  =   & 
    ( { I}_b^{1-\alpha,  \sigma, \varphi}  f )  (x,y)   (1-\sigma) +
\frac{     {}^{\psi}\overline{\partial}_r ({ I}_b^{1-\alpha,  \sigma, \varphi}  f )  (x,y) 
 }{ D\varphi(x,y) }\sigma.   
\end{align*} 
In particular, for $\sigma=\sum_{k=0}^2\psi_k$ and  $\varphi(x) = \sum_{k=0}^2 x_k$, for all $x=\sum_{k=0}^2\psi_k x_k\in \mathbb R^3$ with $x_0,x_1,x_2\in \mathbb R$, we see that  
\begin{align*}
 ({}_a {I}^{\alpha,\sigma,\varphi} f) (x,y) = &\sum _{i=0}^2 ({}_{a_i} I^{\alpha_i} f) (y_0, \dots, x_i ,\dots, y_2) =: ({}_a {  I}^{\alpha} f) (x,y)  , \\
 ( {I}_b^{\alpha,\sigma,\varphi} f) (x,y) = & \sum _{i=0}^2 (  I_{b_i}^{\alpha_i} f)(y_0, \dots, x_i ,\dots, y_3) =: ( {I}_b^{\alpha} f) (x,y),
\end{align*}
\begin{align*}
{}^{\psi}_a{\overline{\partial}}^{\alpha, \sigma, \varphi }   f   (x,y)  =& 
           {}^{\psi}\overline{\partial} ({}_a{I}^{1-\alpha}  f )  (x,y) =: {}^{\psi}_a{\overline{\partial}}^{\alpha } f(x,y), \\
 {}^{\psi} \overline{\partial}_b^{\alpha, \sigma, \varphi }    f   (x,y)=&
 {}^{\psi}\overline{\partial} ({I}_b^{1-\alpha}  f )  (x,y)=:
{}^{\psi} \overline{\partial}_b^{\alpha} f (x,y) ,   
\end{align*}
are the quaternionic fractional integral and derivatives of $f$ in the Riemann-Liouville sense. 

The operators ${}^{\psi}_a \overline{\partial}^{\alpha }$ and ${}^{\psi} \overline{\partial}_b^{\alpha }$ are the fractional ${\psi}$-Cauchy-Riemann operators since 
\begin{align*}
{}^{\psi}_a{\overline{\partial}}^{\alpha, \sigma, \varphi }    f   (x,y)  =& 
\sum_{i=0}^2   
\psi_i ( {}_{a_i}D^{\alpha_i}f)  (y_0,\dots, x_i,\dots, y_2) \\
 = &\sum_{k=0}^3\psi_i \psi_k
\sum_{i=0}^2   
( {}_{a_i}D^{\alpha_i}f_k)  (y_0,\dots, x_i,\dots, y_2)  
,\\ 
{}^{\psi}{\overline{\partial}}_b^{\alpha, \sigma, \varphi }    f   (x,y)  =& 
\sum_{i=0}^2   
 \psi_i  (D^{\alpha_i}_{b_i}f)  (y_0,\dots, x_i,\dots, y_2) \\
	= & \sum_{k=0}^3\psi_i  \psi_k
\sum_{i=0}^2   
( D^{\alpha_i}_{b_i}f_k)  (y_0,\dots, x_i,\dots, y_2),  
\end{align*}
where $f=\sum_{k=0}^2\psi_k f_k$ with $f_0,f_1,f_2$ are taken to be $\mathbb R$-valued functions. Here and subsequently ${}_{a_i}D^{\alpha_i}f_k$ ($D^{\alpha_i}_{b_i}f_k$) stands for the left (respectively right) fractional derivatives of $f_k$ with respect the real component $x_k$ of $x$ for all $i, k=0,1,2$.
\end{definition}

\begin{definition}  Let us fix $\alpha,\sigma, \varphi$, an structural set $\psi$ and a domain $J_a^b$. Assume  $f \in AC^1(J_a^b,\mathcal A)$ be such that $x\mapsto ({}_a{ I}^{\alpha,\sigma, \varphi}f)(x,y)$ belong  to $ C^1(\overline{J_a^b}, \mathcal A)$ for all $y\in J_a^b$. 
\begin{enumerate} 
\item We call $f$ left-$\alpha$-$\varphi$-fractional $\sigma$-proportional $\psi$-hyperholomorphic function on  ${J_a^b }$ ($\alpha$-$\varphi$-$\sigma$-$\psi$-hyperholomorphic function on  ${J_a^b }$, for brevity) if
$$({}^{\psi}_a{\overline{\partial}}^{\alpha,\sigma, \varphi}  f )(x,y)  =0,$$  
for all $x,y\in J_a^b$. 
\item We say that $f$ is right-$\alpha$-$\varphi$-fractional $\sigma$-proportional $\psi$-hyperholomorphic function on  ${J_a^b }$ (r-$\alpha$-$\varphi$-$\sigma$-$\psi$-hyperholomorphic function on ${J_a^b}$, for short) if.
$$({}^{\psi}_a{\overline{\partial}}_r^{\alpha,\sigma, \varphi}f)(x,y),$$  
for all $x,y\in J_a^b$. 
\end{enumerate}
\end{definition}
We will denote by ${}^{\psi}_a{{\mathcal M}}^{\alpha,\sigma, \varphi} (J_a^b)$ (respectively ${}^{\psi}_a{{ \mathcal M}}_r^{\alpha,\sigma, \varphi} (J_a^b)$) the set of all $\alpha$-$\varphi$-fractional $\sigma$-proportional $\psi$-hyperholomorphic functions on ${J_a^b }$ (respectively the set of all r-$\alpha$-$\varphi$-$\sigma$-$\psi$-hyperholomorphic functions on  ${J_a^b}$). Any class of $\alpha$-$\varphi$-fractional $\sigma$-proportional $\psi$-hyperholomorphic functions (r-$\alpha$-$\varphi$-$\sigma$-$\psi$-hyperholomorphic functions) is right (respectively, left) module of $C^1(J_a^b,\mathcal A)$.
\begin{remark} With the notations and definitions introduced previously, let $f\in {}^{\psi}_a{{\mathcal M}}^{\alpha,\sigma, \varphi} (J_a^b)$ then  
 $$ {}^{\psi}\overline{\partial} ({}_a{I}^{1-\alpha,  \sigma, \varphi}  f )  (x,y) + D\varphi(x,y) \sigma^{-1} (1-\sigma) \
   ({}_a{  I}^{1-\alpha,  \sigma, \varphi}  f )  (x,y) =0.$$ 
In fact, ${}_a{I}^{1-\alpha,  \sigma, \varphi} f(x,y)$ is a solution of the perturbed $\psi$-Cauchy-Riemann operator  
$${}^{\psi}\overline{\partial} + D\varphi(x,y) \sigma^{-1} (1-\sigma).$$   

Using  the quaternionic conjugation  we see that 
\begin{align*}
\overline{( {}^{\psi}_a{\overline{\partial}}^{\alpha, \sigma, \varphi }    f )  (x,y) }  = &  
      ({}_a{  I}^{1-\alpha,  \sigma, \varphi}  {\bar f} ) (1-\overline{\sigma} )  (x,y)  +
     \frac{        {}^{\bar \psi}\overline{\partial}_r ({}_a{I}^{1-\alpha,  \sigma, \varphi} \bar{ f} )  (x,y)}{ D\varphi(x,y) } \overline{ \sigma}    ,\end{align*} 
Therefore, 
 $ f\in   {}^{\psi}_a{{\mathcal M}}^{\alpha,\sigma, \varphi} (J_a^b)$ if and only if 
$$      ({}_a{  I}^{1-\alpha,  \sigma, \varphi}  {\bar f} ) (1-\overline{\sigma} )  (x,y)  +
     \frac{        {}^{\bar \psi}\overline{\partial}_r ({}_a{I}^{1-\alpha,  \sigma, \varphi} \bar{ f} )  (x,y)}{ D\varphi(x,y) } \overline{ \sigma} =0, $$
for all $x,y\in J_a^b$.
\end{remark}

\begin{proposition}\label{propFRACD}
Let $f \in AC^1(J_a^b,\mathbb H)$ and ${\alpha} = (\alpha_0, \alpha_1,\alpha_2) \in\mathbb C^3$ with $0< \Re \alpha_{i} <1$ for $i=0,1,2$. Then 
$$\displaystyle {}^{\bar \psi}\overline{\partial}_x \circ( {}^{\psi}\overline{\partial}_a^{{\alpha}} f) (x,y) = 
	\Delta_{\mathbb R^3} ({}_a{I}^{1-\alpha}f)(x,y).$$
\end{proposition}
\begin{proof}   
\begin{align*} 
{}^{\bar \psi}\overline{\partial}_x \circ {}^{\psi}\overline{\partial}_a^{   {\alpha}} f (x,y) =
&{}^{\bar \psi}\overline{\partial}_x \circ  {}^{\psi}\overline{\partial}_x ({}_a{I}^{1-\alpha}  f )  (x,y)  = \Delta_{\mathbb R^3} ({}_a{I}^{1-\alpha}  f )  (x,y) .
\end{align*}
\end{proof} 

\begin{remark}  
Taking $f\in {}^{\psi}_a{{\mathcal M}}^{\alpha,\sigma, \varphi} (J_a^b)$ yields  $\Delta_{\mathbb R^3} ({}_a{I}^{1-\alpha}f)(x,y) =0$. 
\end{remark}

Note that the real proportional fractional derivative  $ {}_aD^{\alpha, \rho,\varphi}$  with respect to a  function  given in 
  \eqref{DerFracPropRespecFuntL} is structurally extended to $\mathbb R^3$ by ${}^{\psi}_a{\overline{\partial}}^{\alpha, \sigma, \varphi }$.
	
\begin{proposition} \label{pro2}
Let $\lambda_0, \dots, \lambda_3 \in C^1(\overline{J_a^b },  \mathbb R)$  be such that  
\[\displaystyle \sum_{k=0}^2 \psi_k \frac{\partial {\lambda_k}}{\partial x_k} (x) =  
 D\varphi(x,y)
  \ \sigma^{-1} (1-\sigma).\]  Then 
\begin{align*}
  {}^{\psi}\overline{\partial}[ e^{ \sum_{k=0}^3\lambda_k(x)}  ({}_a{I}^{1-\alpha,\sigma, \varphi}  f )  (x,y)]  =  &   
    D\varphi  (x,y)  \ 
 \sigma^{-1} {}^{\psi}_a{\overline{\partial}}^{\alpha, \sigma, \varphi }[   f]  (x,y)     e^{ \sum_{k=0}^3 \lambda_k(x)} ,\\
  {}^{\psi}\overline{\partial}_r[ e^{ \sum_{k=0}^3\lambda_k(x)}  ({}_a{I}^{1-\alpha,\sigma, \varphi}  f )  (x,y)]  =  & {}^{\psi}_a{\overline{\partial}}_r^{\alpha, \sigma, \varphi }[   f]  (x,y) \sigma^{-1} \   
	D\varphi (x,y )  \
      e^{ \sum_{k=0}^3 \lambda_k(x)} .
  \end{align*} 
\end{proposition}
\begin{proof}  
\begin{align*}
 & {}^{\psi}\overline{\partial}[ e^{ \sum_{k=0}^3\lambda_k(x)}  ({}_a{I}^{1-\alpha,\sigma, \varphi}  f )  (x,y)]  =   
		\sum_{k=0}^2{\psi_k}\frac{\partial}{\partial x_k}
		[ e^{ \sum_{k=0}^3\lambda_k(x)}  ({}_a{I}^{1-\alpha,\sigma, \varphi}  f )  (x,y)] \\
= & e^{ \sum_{k=0}^3\lambda_k(x)}
						\left[ 						\sum_{k=0}^2{\psi_k}\frac{\partial \lambda_k(x)}{\partial x_k}  ({}_a{I}^{1-\alpha,\sigma, \varphi}  f )  (x,y) +
		\sum_{k=0}^2{\psi_k}\frac{\partial  }{\partial x_k}
		({}_a{I}^{1-\alpha,\sigma, \varphi}  f )  (x,y)
		\right] \\
		= & e^{ \sum_{k=0}^3\lambda_k(x)}
						\left[ 						 D\varphi(x,y)
  \ \sigma^{-1} (1-\sigma)  ({}_a{I}^{1-\alpha,\sigma, \varphi}  f )  (x,y) +
		{}^{\psi}{\overline{\partial}}
		({}_a{I}^{1-\alpha,\sigma, \varphi}  f )  (x,y)
		\right] \\
		= & e^{ \sum_{k=0}^3\lambda_k(x)}
						 D\varphi(x,y)
  \ \sigma^{-1}	\left[ 					 (1-\sigma)  ({}_a{I}^{1-\alpha,\sigma, \varphi}  f )  (x,y) +
	\sigma	\frac{{}^{\psi}{\overline{\partial}}
		({}_a{I}^{1-\alpha,\sigma, \varphi}  f )  (x,y)}{ D\varphi(x,y)}
		\right] .
 \end{align*}
In the same manner we can prove the second identity. 
\end{proof}

\section{Major Achievements}
This section is devoted to prove a quaternionic Borel-Pompieu type formula involving proportional fractional $\psi$-Cauchy-Riemann operators with respect to $\mathbb R$-valued functions.

Fix $\alpha=(\alpha_0,\alpha_1,\alpha_2), \beta=(\beta_0,\beta_1,\beta_2)\in \mathbb C^3$ and let $\sigma:=\sum_{k=0}^2\psi_k \sigma_k$, $\rho =\sum_{k=0}^2\psi_k \rho_k \in \mathbb R^3$ such that $0<\Re \alpha_k, \Re \beta_k$ and $\sigma_k, \rho_k < 1$ for all $k=0,1,2$. 

Consider $\varphi, \vartheta\in C^1(\overline{J_a^b},\mathbb R)$ with $\varphi_i', \vartheta _i' > 0$ on $ [a_i,b_i]$ for all $i=0,1,2$.

\begin{proposition}\label{STF}  
Let $f,g \in AC^1(J_a^b,\mathcal A)$ in such a way that the mappings 
$$x\mapsto ({}_a{  I}^{1-\alpha,\sigma, \varphi}  f )(x,y), \quad   x\mapsto ({}_a{  I}^{1-\beta,\rho, \vartheta} g)(x,y)$$ 
belong to $C^1(\overline{J_a^b}, \mathcal A)$. Then 
\begin{align*}
 & \int_{\partial J_a^b}   ({}_a{  I}^{1-\beta,\rho,\vartheta}  g )  (x,y) 
    \ {}^\psi\eta_x^{\mu,\lambda}  \  ({}_a{ I}^{1-\alpha,\sigma, \varphi}  f ) 
		(x,y) \\
  =  
   &   \int_{J_a^b} \left[   ({}_a{
	I}^{1-\beta,\rho,\vartheta}  g )  (x,y)   
 \   C_{\varphi, \sigma}(x,y) \ ({}^{\psi}_a{\overline{\partial}}^{\alpha,\sigma,\varphi}  f )(x,y)   \right. \\
& \left.  +  ({}^{\psi}_a{\overline{\partial}}_r^{\beta,\rho,\vartheta}  g )(x,y)  \ C_{r,\vartheta, \rho}(x,y) \ 
({}_a{ I}^{1-\alpha,\sigma,\varphi}  f )  (x,y)
\right]    dx^{\mu,\lambda},
\end{align*}  
where  
\begin{align*}
\displaystyle C_{\varphi, \sigma}(x,y):= & D\varphi(x,y)   \sigma ^{-1} , \  \    
 \displaystyle  C_{r,\vartheta, \rho}(x,y):=  \rho ^{-1} D\vartheta (x,y)   , \\ 
 \displaystyle {}^\psi\eta_x^{\mu, \lambda} := &  e^{ \sum_{k=0}^2 
\mu_k(x)+ \lambda_k(x)}\eta^\psi_x, \  \  \displaystyle 
dx^{\mu,\lambda} := e^{\sum_{k=0}^2
 \mu_k(x)+\lambda_k(x)}dx.
\end{align*}
\end{proposition} 
\begin{proof}
Let $\lambda_0, \dots, \lambda_2,\mu_0, \dots, \mu_2  \in C^1(J_a^b , \mathbb R)$ in order to  
\begin{align*} 
\sum_{k=0}^2 \psi_k \frac{\partial {\lambda_k}}{\partial x_k} (x) = & D\varphi(x,y)\sigma^{-1} (1-\sigma)\\ 
\sum_{k=0}^2 \psi_k \frac{\partial {\mu_k}}{\partial x_k} (x) = & D\vartheta(x,y) (1-\rho)\rho^{-1}.\end{align*}
Using  these identities, Proposition \ref{pro2} and applying Stokes formula \eqref{StokesT} for the  functions
$$x\mapsto   e^{ \sum_{k=0}^2 \mu_k(x) }({}_a{I}^{1-\beta,\rho,\vartheta}  g )  (x,y), \quad 
 x \mapsto     e^{ \sum_{k=0}^2 \lambda_k(x)}       ({}_a{I}^{1-\alpha,\sigma, \varphi}  f )  (x,y) ,$$
    we obtain that
\begin{align*}
 & \int_{\partial J_a^b}   ({}_a{  I}^{1-\beta,\rho,\vartheta}  g )  (x,y)     e^{ \sum_{k=0}^3 \mu_k(x)+ \lambda_k(x)} \ {}^\psi\eta_x     ({}_a{  I}^{1-\alpha,\sigma, \varphi}  f )  (x,y) \\
  =  
   &   \int_{J_a^b } \left[   ({}_a{ I}^{1-\beta,\rho,\vartheta}  g )  (x,y)   
   D\varphi(x,y) \sigma ^{-1}({}^{\psi}_a{\overline{\partial}}^{\alpha,\sigma,\varphi}  f )(x,y)   \right. \\
& \left.  +  
({}^{\psi}_a{\overline{\partial}}_r^{\beta,\rho,\vartheta}  g )(x,q)  \rho^{-1} 
 D\vartheta(x,y)      
  ({}_a{ I}^{1-\alpha,\sigma,\varphi}  f )  (x,q)
\right]   e^{\sum_{k=0}^3   \mu_k(x)+\lambda_k(x) } dx.
\end{align*}
\end{proof}

\begin{proposition}\label{BPTF}
Let $f,g \in AC^1(J_a^b,\mathcal A)$ so that the mappings 
$$x\mapsto ({}_a{ I}^{1-\alpha,\sigma, \varphi}  f )(x,y), \quad  x\mapsto ({}_a{  I}^{1-\beta,\rho, \vartheta}  g )(x,y)$$ 
belong to $C^1(\overline{J_a^b}, \mathcal A)$.  Then 
{\begin{align*}  
 &  \int_{\partial J_a^b} K_{\psi}^{\alpha, \sigma, \varphi}(\zeta,x)
 \ {}^{\psi}\eta_{\zeta}  
  ({}_a{  I}^{1-\alpha,\sigma, \varphi}  f )  (\zeta,y)   
  +
  ({}_a{ I}^{1-\beta,\rho, \vartheta}  g )  (\zeta,y) 
		{}^{\psi}\eta_{\zeta}    K_{\psi}^{\beta, \rho, \vartheta}(\zeta,x)
  \\
  & - \sum_{i=0}^2 {}_{a_i}D^{1-\alpha_i,\sigma_i, \varphi_i} \left[
\int_{J_a^b} \mathcal H_{\psi}^{\varphi, \sigma} (\zeta,x)   
({}^{\psi}_a{\overline{\partial}}^{\alpha,\sigma,\varphi}  f )(\zeta,y)    
   d\zeta  \right]  \\
   & - \sum_{i=0}^2 {}_{a_i}D^{1-\beta_i,\rho_i, \vartheta_i} \left[
\int_{J_a^b} ({}^{\psi}_a{\overline{\partial}}_r^{\beta,\rho,\vartheta}  g )(\zeta,y) \mathcal H_{r,\psi}^{\vartheta,\rho}(\zeta,x)    d\zeta  \right]  \\
		=  &  \left\{ 
		\begin{array}{ll}
		 \sum_{i=0}^2 (f+g)(y_0,\dots, x_i,\dots, y_2) 
		+ R^{\alpha, \sigma, \varphi}(f,x,y)   
		  + R^{\beta, \rho, \vartheta}(g,x,y)   , &  x\in J_a^b, 
		 \\ 0 , &  x\in \mathcal A\setminus\overline{J_a^b}.                     
\end{array} \right. 
\end{align*}}

Here ${}_{a_i}D^{1-\alpha_i,\sigma_i,\varphi_i}$ denotes the fractional proportional partial derivative in coordinate $x_i$ with respect to $\varphi_i$ for $i=0,1,2$ and 
\begin{align*}
K_{\psi}^{\alpha, \sigma, \varphi}(\zeta,x) := & \sum_{i=0}^2
 {}_{a_i}D^{1-\alpha_i,\sigma_i, \varphi_i} 
 \left[ K_{\psi}(\zeta-x)  e^{ \sum_{k=0}^2(\lambda_k(\zeta)-
\lambda_k(x) )}\right],\\
\mathcal H_{\psi}^{\varphi, \sigma} (\zeta,x)  := &  K_{\psi} (\zeta-x)     
  e^{   \sum_{k=0}^2 (\lambda_k(\zeta) - \lambda_k(x) ) } 
	C(\varphi, \sigma)(\zeta), \\
\mathcal H_{r,\psi}^{\vartheta,\rho}(\zeta,x) := &	C_r(\vartheta, \rho)(\zeta)
    e^{   \sum_{k=0}^2 (\mu_k(\zeta) - \mu_k(x) ) }  
		K_{\psi} (\zeta-x) ,\\
	R^{\alpha, \sigma, \varphi}(f,x,y) = : &
 \sum_{i=0}^2  ({}_{a_i}{ I}^{1-\alpha_i,\sigma_i, \varphi_i}  f )(y_0, \dots, x_i, \dots, y_2) 
		  \left[ \sum_{j=0, j\neq i}^2 
			{}_{a_j}D^{1-\alpha_j,\sigma_j, \varphi_j}(1)  \right],
			\end{align*}
for all $x \in J_a^b$ and $\zeta \in \overline{J_a^b}$.  
 \end{proposition} 
\begin{proof}
Using  \eqref{BP},  the function  
$x \mapsto e^{ \sum_{k=0}^3\lambda_k(x)} 
 ({}_a{  I}^{1-\alpha,\sigma, \varphi}  f )  (x,y)$ and 
Proposition \ref{pro2}
we obtain  that  
\begin{align*}  
 &  \int_{\partial J_a^b}
K_{\psi}(\zeta-x) \ {}^{\psi}\eta_{\zeta} \   
 e^{ \sum_{k=0}^3 \lambda_k(\zeta)} 
 ({}_a{ I}^{1-\alpha,\sigma, \varphi}  f )  (\zeta,y)   
  \\
  & - 
\int_{J_a^b}  K_{\psi} (\zeta-x)     
  e^{   \sum_{k=0}^3 \lambda_k(\zeta)  } 
	D\varphi (\zeta,y) \sigma^{-1}
	({}^{\psi}_a{\overline{\partial}}^{\alpha,\sigma,\varphi}  f )(\zeta,y) 
   d\zeta   \nonumber \\
		=  & 
		 \left\{ 
		 \begin{array}{ll}  
		    e^{ \sum_{k=0}^3 \lambda_k(x  )} 
				({}_a{ I}^{1-\alpha,\sigma, \varphi}  f )  (x,y) 
						, &  x\in J_a^b,  \\ 
		  0 , &  x\in \mathcal A\setminus\overline{J_a^b}.                     
\end{array} \right. 
\end{align*} 
Multiplication by the factor $e^{ \sum_{k=0}^3 \lambda_k(x)}$ and the action of the fractional differential operator    
 $\sum_{i=0}^3 {}_{a_i}D^{1-\alpha_i,\sigma_i, \varphi_i}$, in terms of real components of $x$, on both sides yields 
\begin{align*}  
 & \sum_{i=0}^3 {}_{a_i}D^{1-\alpha_i,\sigma_i, \varphi_i}
\left[ \int_{\partial J_a^b}
K_{\psi}(\zeta-x)    
 e^{ \sum_{k=0}^3 (\lambda_k(\zeta)- \lambda_k( x))} 
 \ {}^{\psi}\eta_{\zeta} \ ({}_a{ I}^{1-\alpha,\sigma, \varphi}  f )  (\zeta,y)  \right] 
  \\
  & - \sum_{i=0}^3 {}_{a_i}D^{1-\alpha_i,\sigma_i, \varphi_i}\left[
\int_{J_a^b}  K_{\psi} (\zeta-x)     
  e^{   \sum_{k=0}^3( \lambda_k(\zeta) -\lambda_k(x))  } 
	D\varphi (\zeta,y) \sigma^{-1}
	({}^{\psi}_a{\overline{\partial}}^{\alpha,\sigma,\varphi}  f )(\zeta,y) 
   d\zeta \right]  \nonumber \\
		=  & 
		 \left\{ 
		 \begin{array}{ll}  
		    			\sum_{i=0}^3 {}_{a_i}D^{1-\alpha_i,\sigma_i, \varphi_i} \circ 
							({}_a{ I}^{1-\alpha,\sigma, \varphi}  f )  (x,y) 
						, &  x\in J_a^b,  \\ 
		  0 , &  x\in \mathcal A\setminus\overline{J_a^b}.                     
\end{array} \right. 
\end{align*} 
Leibniz rule and the following identity 
\[\sum_{i=0}^3 {}_{a_i}D^{1-\alpha_i,\sigma_i, \varphi_i} \circ 
	({}_a{ I}^{1-\alpha,\sigma, \varphi}  f )  (x,y) 
=		 \sum_{i=0}^3 f(y_0,\dots, x_i,\dots, y_2) 
+ R^{\alpha, \sigma, \varphi}(f,x,q),\]
obtained from \eqref{FundTheoFPW}, give us a part of the Borel-Pompeiu formula. The terms related with $g$ and $({}^{\psi}_a{\overline{\partial}}_r^{\beta,\rho,\vartheta}g)$ follow from similar computations as above.
\end{proof}

We are now in a position  to show the Cauchy theorem and formula for functions in ${}^{\psi}_a{{\mathcal M}}^{\alpha,\sigma, \varphi} (J_a^b)$ and  ${}^{\psi}_a{{\mathcal M}}_r^{\beta,\rho, \vartheta} (J_a^b)$. 
\begin{corollary}  If $f\in {}^{\psi}_a{{\mathcal M}}^{\alpha,\sigma, \varphi} (J_a^b)$ and $g\in {}^{\psi}_a{{\mathcal M}}_r^{\beta,\rho, \vartheta} (J_a^b)$ then 
\begin{align*}
   \int_{\partial J_a^b}   ({}_a{  I}^{1-\beta,\rho,\vartheta}  g )  (x,y) 
    \ {}^\psi\eta_x^{\mu,\lambda}  \  ({}_a{ I}^{1-\alpha,\sigma, \varphi}  f ) 
		(x,y)  
  =  0 \end{align*}  
and   
  \begin{align*}  
 &  \int_{\partial J_a^b} K_{\psi}^{\alpha, \sigma, \varphi}(\zeta,x)
 \ {}^{\psi}\eta_{\zeta}  
  ({}_a{  I}^{1-\alpha,\sigma, \varphi}  f )  (\zeta,y)   
  +
  ({}_a{ I}^{1-\beta,\rho, \vartheta}  g )  (\zeta,y) 
		{}^{\psi}\eta_{\zeta}    K_{\psi}^{\beta, \rho, \vartheta}(\zeta,x)
   \\
		=&  	 \sum_{i=0}^2 (f+g)(y_0,\dots, x_i,\dots, y_2) 
		+ R^{\alpha, \sigma, \varphi}(f,x,y)   		  + R^{\beta, \rho, \vartheta}(g,x,y)   , 
\end{align*} 
for all $x\in   J_a^b $.   
\end{corollary} 
\begin{proof}
The proof strongly depended on the usage of Stokes-type formula induced by ${}^{\psi}_a{ \overline{\partial }}^{\alpha,\sigma, \varphi}$ and   
 ${}^{\psi}_a{ \overline{\partial}}_r^{\beta,\rho, \vartheta}$ together with Proposition \ref{STF}.
\end{proof}

\section*{Statements and Declarations}
\subsection*{Funding} This work was partially supported by Instituto Polit\'ecnico Nacional (grant numbers SIP20232103, SIP20230312) and CONACYT.
\subsection*{Competing Interests} The authors declare that they have no competing interests regarding the publication of this paper.
\subsection*{Author contributions} All authors contributed equally to the study, read and approved the final version of the submitted manuscript.
\subsection*{Availability of data and material} Not applicable
\subsection*{Code availability} Not applicable
\subsection*{ORCID}
\noindent
Jos\'e Oscar Gonz\'alez-Cervantes: https://orcid.org/0000-0003-4835-5436\\
Juan Bory-Reyes: https://orcid.org/0000-0002-7004-1794

\end{document}